\documentclass[reqno, 11pt]{amsart}
\usepackage{bbm}
\usepackage{url}
\usepackage{stmaryrd}
\usepackage{mathrsfs}
\usepackage{cases}
\usepackage{amsfonts}
\usepackage{amssymb}
\usepackage{amsmath}
\usepackage{tikz}
\usepackage{extarrows}

\allowdisplaybreaks[4]
\newtheorem{theorem}{Theorem}[section]
\newtheorem{lemma}[theorem]{Lemma}
\newtheorem{proposition}[theorem]{Proposition}
\newtheorem{corollary}[theorem]{Corollary}
\theoremstyle{definition}
\newtheorem{definition}[theorem]{Definition}

\newtheorem{remark}[theorem]{Remark}
\numberwithin{equation}{section}

\let\al=\alpha
\let\b=\beta
\let\g=\gamma

\let\La=\Lambda

\let\Om=\Omega

\def\bbR{\mathbb{R}}
\def\bbS{\mathbb{S}}

\newcommand{\be}{\begin{equation*}}
\newcommand{\ee}{\end{equation*}}
\newcommand{\ben}{\begin{equation}}
\newcommand{\een}{\end{equation}}
\newcommand{\bn}{\begin{enumerate}}
\newcommand{\en}{\end{enumerate}}

\def\tA{\widetilde{A}(y)}
\def\dA{\det{A(y)}}

\def\bmo{{{\rm BMO}(\mathbb R^n)}}
\def\wk1{W^{k,1}(\mathbb{R}^n)}

\begin{document}
\title[Hausdorff operators on Sobolev space]{Hausdorff operators on Sobolev space $W^{k,1}$}
\author{GUOPING ZHAO}
\address{School of Applied Mathematics, Xiamen University of Technology, Xiamen, 361024, P.R.China} \email{guopingzhaomath@gmail.com}
\author{WEICHAO GUO}
\address{School of Mathematics and Information Sciences, Guangzhou University, Guangzhou, 510006, P.R.China}
\email{weichaoguomath@gmail.com}

\begin{abstract}
  This paper is served as a first contribution regarding the boundedness of Hausdorff operators on function spaces with smoothness.
  The sharp conditions are established for boundedness of Hausdorff operators on Sobolev spaces $W^{k,1}$.
  As applications, some bounded and unbounded properties of Hardy operator and adjoint Hardy operator on $W^{k,1}$ are deduced.
\end{abstract}

\maketitle

\section{Introduction and Preliminary}
Sobolev spaces is one of the most important function spaces in the fields of Harmonic analysis and partial differential equations.
For an integer $k\geq 0$, and $1\leq p <\infty$,
the classical Sobolev space $W^{k,p}(\mathbb{R}^n)$ is defined as the space of functions $f$,
with $f\in L^p(\mathbb{R}^n)$ and all derivatives, denoted by $D^{\al}f$, exist in the weak sense and
belong to $L^p(\mathbb{R}^n)$ for all $|\alpha|\leq k$.
The corresponding norm for the function space $W^{k,p}:= W^{k,p}(\mathbb{R}^n)$ is defined by
\begin{equation*}
  \|f\|_{W^{k,p}}:=\sum_{|\alpha|\leq k} \left\| D^{\al}f \right\|_{L^p},
  \left(D^{0}f=f\right).
\end{equation*}
With the norm defined above, $W^{k,p}$ is a Banach space.
One can see \cite{Stein_book_singular-integral} for a nice description of the basic definitions and properties about Sobolev spaces.

As we know, in the field of harmonic analysis and PDE, it is quite important to study whether the regularity of a function (or the initial data)
can be persisted through certain operator, for instance, the boundedness of certain operators on Sobolev spaces.
One can see the celebrated Sobolev embedding theorem in \cite{Stein_book_singular-integral},
see \cite{CMP_JFA_2017, Liu-Chen-Wu_Bulletin Australian_2016} for the boundedness of Hardy-Littlewood maximal operator on Sobolev spaces,
and see \cite{Brezis_book} for the study of evolution equations on Sobolev spaces.

Note that the previous works were mainly concerned with convolution operators, it is of great interest to
consider how the regularity can be transferred through certain operators of non convolution type.
Here, we introduce such a class of operators, named Hausdorff operators.
For a suitable function $\Phi$, one of the corresponding Hausdorff operators $H_{\Phi}$ is defined by
\begin{equation}
  H_{\Phi}f(x):=\int_{\mathbb{R}^n}\Phi(y)f\left(\frac{x}{|y|}\right)\mathrm{d}y,
\end{equation}
where the above integral makes sense for $f$ belongs to some classes of nice functions. Obviously, $H_{\Phi}$
is not a convolution operator, and does not commute with translations.

The study of Hausdorff operators, which originated from some classical summation methods, has a long history in real and complex analysis.
The interested reader can refer to
\cite{Chen_Fan_Wang_2013} and \cite{Liflyand_survey} for a survey of some historical background and recent developments
regarding Hausdorff operators.
Particularly, Hausdorff operator is an interesting operator in harmonic analysis.
It contains some important operators when $\Phi$ is taken suitably,
such as Hardy operator, adjoint Hardy operator (see \cite{Chen_Fan_Li_2012, Chen_Fan_Zhang_2012, Fan_Lin_Analysis_2014}),
and the Ces\`{a}ro operator \cite{Miyachi_2004, Siskakis} in one dimension.
The Hardy-Littlewood-P\'{o}lya operator and the Riemann-Liouville fractional integral can also be derived from the Hausdorff operator.

In recent years, there is an increasing interest on the
the boundedness of Hausdorff operators on function spaces,
see for example \cite{Gao-Wu-Guo_Math.Inequal.Appl_2015,Gao-Zhong_Math.Inequal.Appl_2014,Liflyand-Miyachi_Studia,Liflyand-Moricz_Proceeding}).
However, the boundedness of Hausdorff operator can be characterized in only few cases.
We refer the reader to \cite{Gao-Zhao_Anal.Math_2015, Wu_Chen_SCIChina_2014} for the characterization
of the bounded Hausdorff operators on Lebesgue spaces,
to \cite{Fan_Lin_Analysis_2014, Ruan_Fan_JMAA_2016} for the characterization of the bounded Hausdorff operators on
Hardy spaces $H^1$ and $h^1$, and to \cite{Zhao_Guo_AFA} for the characterization of the bounded Hausdorff operators on modulation and Wiener amalgam spaces. Until now, there is no result regarding the boundedness of Hausdorff operators on Sobolev spaces.
One of our motivations is to serve as a first contribution for this.

In this paper, we consider the $W^{k,1} (k\geq 1)$ boundedness of
generalized Hausdorff operators $H_{\Phi,A}$ defined by
\begin{equation*}
  H_{\Phi,A}f(x):=\int_{\mathbb{R}^n}\Phi(y)f(A(y)x)\mathrm{d}y,
\end{equation*}
where $A:=\left(a_{ij}(y)\right)_{n\times n}$ is a matrix depending on the variable $y$.
The operator $H_{\Phi,A}$ was first studied by M\'{o}ricz \cite{Moricz2005AnalMath} and Lerner-Liflyand \cite{LernerLiflyand2007jaums},
and it is easy to see that $H_{\Phi,A}=H_{\Phi}$ when $A(y)=\text{diag}\{1/|y|,\cdots,1/|y|\}$.

For $x=(x_1,x_2,\cdots,x_n)\in\bbR^n$, $|x|:=\left(\sum\limits_{i=1}^n x_i^2\right)^{1/2}$,
and for a multi-index $\alpha=(\al_1,\al_2,\cdots,\al_n)$ with $\al_i\in\mathbb{N}$, $|\al|:=\sum\limits_{i=1}^n \al_i$.
For a matrix $B=(b_{ij})_{m\times n}$, we set the norm by
\begin{equation*}
  \|B\|:=\left(\sum_{i=1}^m\sum_{j=1}^n|b_{ij}|^2\right)^{1/2}.
\end{equation*}
Note that this norm is equivalent to the operator norm defined by
\be
\|B\|_{op}:=\sup\limits_{x\in \bbR^n, |x|\neq 0}\frac{|Bx|}{|x|}.
\ee
Denote by $\det B$ the determination of the matrix $B=(b_{ij})_{m\times n}$ with $m=n$.
Set
$$\mathscr{A}:
=\{
 A(y)|
 A(y)=\Lambda P(y) Q,
 \text{ where }det Q, det \Lambda \neq 0, P(y)\in\mathscr{P}\}
\},
$$
where
$$\mathscr{P}:=\{
 B(y)=(b_{ij}(y))_{n\times n} | b_{ij}(y)\geq0 \text{ uniformly on }  \mathrm{supp}\,\Phi\}.$$
Write the matrix into the form of  column vectors:
$$A(y)=\{\vec{A}_1(y), \vec{A}_2(y), \cdots, \vec{A}_n(y)\}.$$

Comparing with the convolution operator such as $T_{\Phi}f(x):=(\Phi\ast f)(x)=\int_{\bbR^n}\Phi(y)f(x-y)dy$,
the translation operator acted on $f$ is now replaced by a dilation operator in $H_{\Phi,A}$ (see also $H_{\Phi}$).
Note that the derivation operation commutes with translation, but does not commute with dilation.
Thus, in order to establish a bounded result on Sobolev space for $H_{\Phi,A}$,
the conditions on $\Phi$ should be related to the order of smoothness of Sobolev space.
This is one of the main differences between convolution and non convolution operators.

More difficulties come from the lower bound estimates when we establish
the necessity conditions for the Sobolev boundedness of $H_{\Phi,A}$.
To be more specific, we have following two difficulties:
\bn
\item As mentioned before, the conditions on $\Phi$ is related to the smoothness of Sobolev space, see (\ref{condition of main theorem, matrix}).
The type of conditions seems to appear when the derivative operation can be transferred from $H_{\Phi,A}f$ to $f$, see (\ref{derivative of Hausdorff operator in main theorem, matrix}).
However, this is impossible before we know $\Phi$ satisfies the desired conditions.
\item
When dealing with the necessity part, we must choose a suitable $f$ and estimate $\|H_{\Phi,A}f\|_{W^{k,1}}$ from below
by some integral involving $\Phi$.
This is not difficult when $k=0$, i.e., the Lebesgue case, in which we can choose a nonnegative $f$ to get the desired estimates.
However, when $k\geq 1$, the Sobolev norm $\|H_{\Phi,A}f\|_{W^{k,1}}$ in fact contains
some integral terms involving $\Phi$ and $\frac{\partial^{\al}f}{\partial x^{\al}}$ for $|\al|\leq k$.
In order to establish the lower estimates of these terms, $\frac{\partial^{\al}f}{\partial x^{\al}}$ is expected to be nonzero and nonnegative.
However, a function like this does not belong to any Sobolev space.
\en

Now, we state our main results as follows.

\begin{theorem}\label{theorem, boundedness, matrix}
  Suppose that the Hausdorff operator $H_{\Phi,A}$ satisfyies
  \begin{equation}\label{condition of main theorem, matrix}
    \int_{\mathbb{R}^n}\left|\det A(y)\right|^{-1} \left(1+\|A(y)\|^{k} \right) \Phi(y)\mathrm{d}y<\infty
  \end{equation}
  for some integer $k\geq 0$.
  Then, $H_{\Phi}$ is bounded on Sobolev space $W^{k,1}$.
  Moreover, if (\ref{condition of main theorem, matrix}) holds, for $f\in W^{k,1}$ we have
  \begin{align}\label{derivative of Hausdorff operator in main theorem, matrix}
      &D^{\al}(H_{\Phi, A}f)(x)
      =\int_{\bbR^n}\Phi(y)D^{\al}\big(f(A(y)x)\big)\mathrm{d}x
      \nonumber
      \\
      &=\int_{\mathbb{R}^n}\Phi(y)\cdot
      \left[\prod_{j=1}^n\left(\sum_{i=1}^na_{ij}\frac{\partial}{\partial x_i}\right)^{\al_j}\right](f)(A(y)x)
        \mathrm{d}y
    \end{align}
  for all multi-index $\alpha=\{\alpha_1, \alpha_2,\cdots,\alpha_n\}$ with $|\alpha|\leq k$.
  Especially,
  \begin{equation}
    \nabla H_{\Phi,A}f(x)
    =\int_{\mathbb{R}^n}\Phi(y)\nabla f\left(A(y)x\right)\cdot A(y)\mathrm{d}y
    \ \text{ for}\ f\in W^{1,1},
  \end{equation}
  \begin{equation}
    \left(\partial_{ij} H_{\Phi,A}f(x)\right)_{n\times n}
    =\int_{\mathbb{R}^n}\Phi(y)A^T(y)\cdot\left(\partial_{ij} f\left(A(y)x\right)\right)_{n\times n}\cdot A(y)\mathrm{d}y
    \ \text{ for}\ f\in W^{2,1}.
  \end{equation}
\end{theorem}

\begin{theorem}\label{theorem, characterization, matrix}
  Let $\Phi$ be a nonnegative function on $\bbR^n$, and let $k\geq 0$ be an integer.
  Suppose that $A(y)\in\mathscr{A}$ with $|\det A(y)| > \eta \prod\limits_{j=1}^n\|\vec{A}_j(y)\|$ uniformly on $\mathrm{supp}\,\Phi$
  for some $\eta>0$.
  The following three statements are equivalent
  \bn
  \item $H_{\Phi,A}$ is bounded on $W^{k,1}(\mathbb{R}^n)$;
  \item $H_{\Phi,A}f\in W^{k,1}$ if $f\in W^{k,1}$;
  \item $\int_{\mathbb{R}^n}\left|\det A(y)\right|^{-1} \left(1+\|A(y)\|^{k} \right) \Phi(y)\mathrm{d}y<\infty$.
  \en
\end{theorem}
\begin{remark}
  The condition $|\det A(y)| > \eta \prod_{j=1}^n\|\vec{\alpha}_j(y)\|$ in Theorem \ref{theorem, characterization, matrix} is in fact reflect the degree of linear independence for the column vector of the normalized matrix of $A(y)$ (see
  $\widetilde{A}(y)$ in the proof of Theorem \ref{theorem, characterization, matrix}),
  which ensure that an open cone can not be squashed too much.
\end{remark}

Especially, if the matrix $A(y)$ is the diagonal matrix
 $\text{diag}\{1/|y|, 1/|y|,\cdots, 1/|y|\}$, i.e., we have following boundedness characterization of $H_{\Phi}$.

\begin{corollary}\label{corollary, boundedness of Hausdorff operator on Sobolev space L_k^1}
  Let $\Phi$ be a nonnegative function on $\bbR^n$, and let $k\geq 0$ be an integer.
  Then
  $H_{\Phi}$ is bounded on $W^{k,1}(\mathbb{R}^n)$
  if and only if
  \begin{equation}\label{condition of main theorem}
    \int_{\mathbb{R}^n}|y|^n \left(1+|y|^{-k} \right) \Phi(y)\mathrm{d}y<\infty.
  \end{equation}
  Furthermore, if (\ref{condition of main theorem}) holds, we have
  \begin{equation}\label{derivative of Hausdorff operator in main theorem}
      \frac{\partial^{\alpha} H_{\Phi}f}{\partial x^{\alpha}}(x)
      =\int_{\mathbb{R}^n}|y|^{-|\alpha|}\Phi(y)\frac{\partial^{\alpha} f}{\partial x^{\alpha}}\left(\frac{x}{|y|}\right)\mathrm{d}y\ \text{ for}\ f\in W^{k,1}
    \end{equation}
for all $|\alpha|\leq k$.
\end{corollary}
Our article is organized as follows.
Section 2 is devoted to the proof of Theorem \ref{theorem, boundedness, matrix} and \ref{theorem, characterization, matrix}.
As applications, in Section 3 we get some bounded and unbounded properties
for Hardy operator and adjoint Hardy operator.

Throughout this paper, we will adopt the following notations.
$X\lesssim Y$ denotes the statement that $X\leq CY$, with a positive constant $C$ that may depend on $n, \,p$,
but it might be different from line to line.
The notation $X\sim Y$ means the statement $X\lesssim Y\lesssim X$.
We use $X\lesssim_{\lambda}Y$ to denote $X\leq C_{\lambda}Y$,
meaning that the implied constant $C_{\lambda}$ depends on the parameter $\lambda$.

\section{Proof of main theorem }
This section is devoted to the proofs of
Theorem \ref{theorem, boundedness, matrix} and \ref{theorem, characterization, matrix}.
We first recall the definition of weak derivative, which will be used frequently in our proofs.

\begin{definition}(Weak Derivative)
  Let
$D^{\al}
=\frac{\partial^m}{\partial x^{\alpha}}
=\frac{\partial^{\alpha_1+\alpha_2+\cdots+\alpha_n}}{\partial x_1^{\alpha_1}\partial x_2^{\alpha_2}\cdots\partial x_n^{\alpha_n}}$
be a differential monomial, whose total order is $m=|\alpha|=\alpha_1+\alpha_2+\cdots+\alpha_n$,
where $\al_j\in\mathbb{N}$.
Suppose we are given two locally integrable functions on $\mathbb{R}^n$, $f$ and $g$.
Then we say that $D^{\al}f=g$ (in the weak sense), if
\begin{equation}\label{Definition of weak derivatives}
  \int_{\mathbb{R}^n}f(x)D^{\al}\varphi(x)\mathrm{d}x
  =
  (-1)^{|\alpha|} \int_{\mathbb{R}^n}g(x)\varphi(x)\mathrm{d}x,\,\, \text{for all }\varphi\in\mathscr{D},
\end{equation}
where $\mathscr{D}$ is the space of indefinitely differential functions with compact support.
\end{definition}
Integration by parts shows us that this is indeed the relation that we would expect if $f$ had continuous partial derivatives up to order $|\alpha|$, and $D^{\al}f=g$ had the usual meaning.
It is of course not ture that every locally integrable function has partial derivatives in this sense.
However when the partial derivatives exist they are determined almost everywhere by the defining relation (\ref{Definition of weak derivatives}), see \cite{Stein_book_singular-integral}.


\begin{proof}[Proof of Theorem \ref{theorem, boundedness, matrix}]
By (\ref{condition of main theorem, matrix}) , for any  multi-index $\alpha$ with $|\alpha|\leq k$,
$f\in W^{k,1}$ and $\varphi\in \mathscr{D}$, we have
\be
\begin{split}
  \int_{\mathbb{R}^n}\int_{\mathbb{R}^n}|\Phi(y)f\left(A(y)x\right) D^{\al}\varphi(x)| \mathrm{d}x\mathrm{d}y
  \leq
  \|D^{\al}\varphi\|_{L^{\infty}}\|f\|_{L^1}\int_{\bbR^n}|\det(A(y))|^{-1}\Phi(y)\mathrm{d}y<\infty,
\end{split}
\ee
and
\be
\begin{split}
 & \int_{\mathbb{R}^n}\int_{\mathbb{R}^n}
  \left|\Phi(y)\left[\prod_{j=1}^n\left(\sum_{i=1}^na_{ij}(y)\frac{\partial}{\partial x_i}\right)^{\al_j}\right](f)(A(y)x)
        \varphi(x)\right|\mathrm{d}x \mathrm{d}y
  \\
  \leq &
  \|\varphi\|_{L^{\infty}}\|f\|_{W^{k,1}} \int_{\bbR^n}|\det(A(y))|^{-1}\|A(y)\|^{|\al|}\Phi(y)\mathrm{d}y<\infty.
\end{split}
\ee
From the above two equalities we deduce that
\begin{equation}\label{for proof, 10}
  \begin{split}
    &\int_{\mathbb{R}^n}H_{\Phi, A}f D^{\al}\varphi(x)\mathrm{d}x
    =
    \int_{\mathbb{R}^n}\int_{\mathbb{R}^n}\Phi(y)f\left(A(y)x\right)\mathrm{d}y D^{\al}\varphi(x)\mathrm{d}x
    \\
    &=
    \int_{\mathbb{R}^n}\Phi(y)
    \int_{\mathbb{R}^n}
      f\left(A(y)x\right) D^{\al}\varphi(x) \mathrm{d}x \mathrm{d}y
    \\
    &=
        (-1)^{|\al|}
    \int_{\mathbb{R}^n}\Phi(y)
    \int_{\mathbb{R}^n}
      D^{\al}(f(A(y)x))
        \varphi(x)\mathrm{d}x \mathrm{d}y
    \\
    &=
    (-1)^{|\al|}
    \int_{\mathbb{R}^n}\Phi(y)
    \int_{\mathbb{R}^n}
      \left[\prod_{j=1}^n\left(\sum_{i=1}^na_{ij}(y)\frac{\partial}{\partial x_i}\right)^{\al_j}\right](f)(A(y)x)
        \varphi(x)\mathrm{d}x \mathrm{d}y
    \\
    &=
    (-1)^{|\al|}
    \int_{\mathbb{R}^n}
    \int_{\mathbb{R}^n}\Phi(y)
      \left[\prod_{j=1}^n\left(\sum_{i=1}^na_{ij}(y)\frac{\partial}{\partial x_i}\right)^{\al_j}\right](f)(A(y)x)\mathrm{d}y\varphi(x)\mathrm{d}x,
  \end{split}
\end{equation}
where we use the Fubini theorem in the second and last equalities.
This yields that
\begin{equation}\label{for proof, 11}
  D^{\al}(H_{\Phi, A}f)(x)
  =
  \int_{\mathbb{R}^n}\Phi(y)
      \left[\prod_{j=1}^n\left(\sum_{i=1}^na_{ij}(y)\frac{\partial}{\partial x_i}\right)^{\al_j}\right](f)(A(y)x)\mathrm{d}y
\end{equation}
in the weak sense.
Note that (\ref{for proof, 11}) is valid for all multi-index $\alpha$ with $|\al|\leq k$, we conclude that
\begin{equation*}
  \begin{split}
    &\|H_{\Phi, A}f\|_{W^{k,1}}
    =
    \sum_{|\alpha|\leq k}
    \left\|
    D^{\al}(H_{\Phi, A}f)\right\|_{L^1}
    \\
    &=
    \sum_{|\alpha|\leq k}
    \left\|
     \int_{\mathbb{R}^n}\Phi(y)
      \left[\prod_{j=1}^n\left(\sum_{i=1}^na_{ij}(y)\frac{\partial}{\partial x_i}\right)^{\al_j}\right](f)(A(y)\cdot)\mathrm{d}y
    \right\|_{L^1}
    \\
    &\leq
    \sum_{|\alpha|\leq k}
     \int_{\mathbb{R}^n}\Phi(y)
      \left\|\left[\prod_{j=1}^n\left(\sum_{i=1}^na_{ij}(y)\frac{\partial}{\partial x_i}\right)^{\al_j}\right](f)(A(y)\cdot)\right\|_{L^1}\mathrm{d}y
      \\
    &\leq
    \sum_{|\alpha|\leq k}
     \int_{\mathbb{R}^n}\Phi(y)|\det A(y)|^{-1}
      \left\|\left[\prod_{j=1}^n\left(\sum_{i=1}^na_{ij}(y)\frac{\partial}{\partial x_i}\right)^{\al_j}\right]f\right\|_{L^1}\mathrm{d}y
      \\
    &\leq
     \int_{\mathbb{R}^n}\Phi(y)|\det A(y)|^{-1}
      \sum_{|\alpha|\leq k}\left[\prod_{j=1}^n\left(\sum_{i=1}^na_{ij}(y)\right)^{\al_j}\right]\mathrm{d}y
      \|f\|_{W^{k,1}}
      \\
      &\lesssim
       \int_{\mathbb{R}^n}\left|\det A(y)\right|^{-1} \left(1+\|A(y)\|^{k} \right) \Phi(y)\mathrm{d}y\|f\|_{W^{k,1}}.
   \end{split}
\end{equation*}
\end{proof}

In order to give the proof of Theorem \ref{theorem, characterization, matrix},
we first give some basic properties of Gaussian function.
In fact, we will use Gaussian function to construct some useful functions with
nonnegative derivatives on our desired domain of $\bbR^n$.

\begin{lemma}[Derivatives of Gaussian function]\label{lemma, Derivatives of Gaussian function}
  Let $g(t)=\mathrm{e}^{-t^2}$ for $t\in \mathbb{R}$, $g^{(m)}$ be the m-th derivative of $g$. Then
  \begin{enumerate}
    \item $g^{(m)}(t)=P_m(t)g(t)$ where $P_m(t)$ is a $m$-order polynomial with respect to the variable $t$;
    \item the sign of the highest order term of $P_m$ is equal to $(-1)^m$.
  \end{enumerate}
\end{lemma}
\begin{proof}
  Denote by $Q(m)$ the conclusion of this lemma.
  Since $g'(t)=-2t\mathrm{e}^{-t^2}$, $Q(1)$ is correct.

  Next, we assume $Q(l-1)$ holds.
  Write $P_{l-1}(t)$ by
  \begin{equation}
    P_{l-1}(t)=(-1)^{l-1}A_{l-1}t^{l-1}+A_{l-2}t^{l-2}+\cdots+A_0,
  \end{equation}
  where $A_{l-1}>0$, $A_j\in\mathbb{R}$ for $j=0,1,\cdots ,l-2$.
  Then by derivative formula of multiplication we can calculate $g^{(l)}$ by
  \begin{equation*}
    \begin{split}
      g^{(l)}(t)=&\left(g^{(l-1)}(t)\right)'=\left(P_{(l-1)}(t)g(t)\right)'=\left(P_{(l-1)}(t)\right)'g(t)+ P_{(l-1)}\left(g(t)\right)'
      \\
      =&
      [(-1)^{l-1}A_{l-1}(l-1)t^{l-2}+(l-2)A_{l-2}t^{l-3}+\cdots+A_1]g(t)
      \\
      &+
      [(-1)^{l-1}A_{l-1}t^{l-1}+A_{l-2}t^{l-2}+\cdots+A_0](-2t)g(t)
      \\
      =&
      [(-1)^{l}2A_{l-1}t^{l}+(-2A_{l-2})t^{l-1}+\cdots+A_1]g(t).
    \end{split}
  \end{equation*}
  Thus, $P_l(t)=(-1)^{l}2A_{l-1}t^{l}+(-2A_{l-2})t^{l-1}+\cdots+A_1$ is a $l$-order polynomial and the leading term is $(-1)^{l}2A_{l-1}t^{l}$.
  Noticing that $A_{l-1}$ is positive, we have the sign of the leading coefficient of $P_l$ is $(-1)^l$.

  By the mathematical induction, $Q(m)$ holds for all $m\in \mathbb{N}$.
\end{proof}

\begin{proof}[Proof of Theorem \ref{theorem, characterization, matrix}]
  Note that $(3)\Longrightarrow (1)$ follows by Theorem \ref{theorem, boundedness, matrix},
  and $(1)\Longrightarrow (2)$ is trivial, we only need to give the proof of $(2)\Longrightarrow (3)$.
  First, we give a reduction for the matrix $A$.
  Write $A(y)=\La P(y)Q$, where $\det\La$, $\det Q\neq 0$,   $P(y)\in \mathscr{P}$.
  For any $f\in \wk1$, let $F(x):=f(\La^{-1}x)$, then $f(x)=F(\La x)$.
  Using this notation, we write
  \be
  \begin{split}
  H_{\Phi,P}(f)(x)
  = &
  \int_{\bbR^n}\Phi(x)f(P(y)x)\mathrm{d}y
  \\
  = &
  \int_{\bbR^n}\Phi(x)F(\La P(y)QQ^{-1}x)\mathrm{d}y
  \\
  = &
  \int_{\bbR^n}\Phi(x)F(A(y)Q^{-1}x)\mathrm{d}y=H_{\Phi,A}(F)(Q^{-1}x).
  \end{split}
  \ee
Thus,
\be
\|H_{\Phi,P}f(\cdot)\|_{W^{k,1}}=\|H_{\Phi,A}(F)(Q^{-1}\cdot)\|_{W^{k,1}}\sim\|H_{\Phi,A}(F)(\cdot)\|_{W^{k,1}},
\ee
This and the following relation
\be
\|f(\cdot)\|_{W^{k,1}}= \|F(\La\cdot)\|_{W^{k,1}} \sim \|F\|_{W^{k,1}}
\ee
yield that
the conclusion (2) is valid for $H_{\Phi,P}$ if it is valid for $H_{\Phi,A}$.
  Without loss of generality, we assume $A\in \mathscr{P}$,
  and $H_{\Phi,A}$ satisfies (2) in the remainder of this proof.

  The proof of $k=0$ is trivial, we assume $k\geq 1$.
  By the definition of $W^{k,1}$, we have $W^{k,1}\subset L^1$.
  Take $f$ to be a nonzero Schwartz function with nonnegative value, then $f\in W^{k,1}$, and $H_{\Phi,A}f\in W^{k,1}\subset L^1$.
  Recall that $A(y)$ is invertible and $\Phi\geq 0$. The Tonelli theorem and  variable substitution yield that
  \begin{equation*}
    \begin{split}
      \|H_{\Phi,A}f\|_{L^1}
      &=\int_{\mathbb{R}^n} \int_{\mathbb{R}^n} \Phi(y) f\left(A(y)x\right) \mathrm{d}y \mathrm{d}x
      \\
      &=\int_{\mathbb{R}^n} \Phi(y)  \int_{\mathbb{R}^n} f\left(A(y)x\right) \mathrm{d}x \mathrm{d}y
      \\
      &=\|f\|_{L^1}\int_{\mathbb{R}^n} \left|\det A(y)\right|^{-1} \Phi(y) \mathrm{d}y<\infty.
    \end{split}
  \end{equation*}
From this we have
  \begin{equation}\label{condition 1 of Phi in L_k^1, matrix}
   \int_{\mathbb{R}^n}\left|\det A(y)\right|^{-1}\Phi(y)\mathrm{d}y<\infty.
  \end{equation}
  We get the first condition of $\Phi$. Next, we proceed to seeking the information of $\Phi$ from smoothness.
  \\
  \textbf{Step 1: first order derivative.}\\
  It follows from the the assumption (2) that $H_{\Phi, A}f \in W^{k,1}$ for any fixed $f\in W^{k,1}$.
  By the definition of $W^{k,1}$, the weak derivative of $H_{\Phi, A}f$ exists and belongs to $L^1$.
  Moreover, for every $\varphi\in\mathscr{D}$,
  \begin{equation}\label{derivative of Hausdorff operator for alpha=1, 1, matrix}
    \begin{split}
      \int_{\mathbb{R}^n}\frac{\partial H_{\Phi, A}f}{\partial x_j}(x)\varphi(x) \mathrm{d}x
      =-\int_{\mathbb{R}^n}H_{\Phi, A}f(x)\frac{\partial \varphi}{\partial x_j}(x)\mathrm{d}x
      =-\int_{\mathbb{R}^n} \int_{\mathbb{R}^n} \Phi(y) f\left(A(y)x\right) \mathrm{d}y \frac{\partial \varphi}{\partial x_j}(x)\mathrm{d}x,
    \end{split}
  \end{equation}
  for $1\leq j\leq n$.
  Combing (\ref{condition 1 of Phi in L_k^1, matrix}) with the fact that $\varphi\in\mathscr{D}$ and $f\in L_k^1\subset L^1$,
  we have
  \begin{equation*}
    \begin{split}
      \int_{\mathbb{R}^n} \int_{\mathbb{R}^n} \left|\Phi(y) f\left( A(y)x \right) \frac{\partial \varphi}{\partial x_j}(x)\right|dxdy
      \leq &
      \left\|\frac{\partial \varphi}{\partial x_j}\right\|_{L^{\infty}}
      \int_{\mathbb{R}^n}|\Phi(y)| \int_{\mathbb{R}^n} \left|f\left( A(y)x \right)\right|dxdy
      \\
      = &
      \left\|\frac{\partial \varphi}{\partial x_j}\right\|_{L^{\infty}}
      \int_{\mathbb{R}^n}\left|\det A(y)\right|^{-1}\Phi(y)\mathrm{d}y\cdot \|f\|_{L^1}<\infty.
    \end{split}
  \end{equation*}
  Hence, we can use the Fubini theorem in (\ref{derivative of Hausdorff operator for alpha=1, 1, matrix}) to deduce that
  \begin{equation}\label{derivative of Hausdorff operator for alpha=1, 2, matrix}
    \begin{split}
      \int_{\mathbb{R}^n}\frac{\partial H_{\Phi, A}f}{\partial x_j}(x)\varphi(x) \mathrm{d}x
      &=-\int_{\mathbb{R}^n} \Phi(y) \int_{\mathbb{R}^n} f\left( A(y)x \right)  \frac{\partial \varphi}{\partial x_j}(x)\mathrm{d}x \mathrm{d}y
      \\
      &
      =\int_{\mathbb{R}^n}\Phi(y)
         \int_{\mathbb{R}^n} \varphi(x)\cdot \sum_{i=1}^n a_{i j}(y)\frac{\partial f}{\partial x_{i}}\left( A(y)x \right)\mathrm{d}x \mathrm{d}y,
    \end{split}
  \end{equation}
  for $1\leq j\leq n$, where in the last inequality we use the definition of weak derivative again.

  Set $g(t):=\mathrm{e}^{-t^2}$, $t\in\bbR$ and $G_1(x):=-\prod\limits_{l=1}^{n}g(x_l+1/2)$, where $x_l$ is the $l$-th component of $x=(x_1,x_2,\cdots,x_n)\in\bbR^n$.
  Note that $G_1$ is a Schwartz function, and belongs to $W^{k,1}$.
  By a direction calculation, we obtain
  \ben\label{for proof, 14}
  \frac{\partial G_1}{\partial x_j}(x)
  =2\left(x_j+\frac{1}{2}\right)\prod\limits_{l=1}^{n}g\left(x_l+\frac{1}{2}\right)
  \geq \mathrm{e}^{-2|x|^2-\frac{n}{2}}>0,\ \ \ x\in \Om:=(0,\infty)^n.
  \een
  Denote $\Om_r:=(0,\infty)^n\cap B(0,r)$, where $B(0,r)$ is the ball with  radius $r$ centered at the origin.
  Especially, $B(0,1)$ is the unit ball centered at origin. Denote by $\mathbb{S}^{n-1}$ the sphere of $B(0,1)$.
  Set
  $$
  \tA:=
  \left\{
  \frac{\vec{A}_1}{n\|\vec{A}_1\|},
  \frac{\vec{A}_2}{n\|\vec{A}_2\|},
  \cdots,
  \frac{\vec{A}_n}{n\|\vec{A}_n\|}
  \right\}
  .$$
  A direct computation yields that $\|\tA\|_{op}\leq \|\tA\|\leq 1$.
  This implies that $\tA\Om_1\subset\Om_1$.
  Set $$\widetilde\Om_1:=\{x\in\Om_1, \text{there exists } \lambda>0 \text{ such that } \lambda x\in\tA\Om_1\}.$$
  We have $A(y)\Om=\tA\Om\subset \Om$.
  Recall that  $|\dA|> \eta \cdot \|\vec{A}_1\|\|\vec{A}_2\| \cdots \|\vec{A}_n\|$ uniformly for $y\in\mathrm{supp\,}\Phi$,
  we have $|\det \tA|> C_{\eta}(=\eta/n^n)$ uniformly for $y\in\mathrm{supp\,}\Phi$.
  Hence,
  \be
  |\tA\Om_1|=\int_{\bbR^n}\chi_{\Om_1}(\tA^{-1}x)\mathrm{d}x=|\det \tA|\int_{\bbR^n}\chi_{\Om_1}(x)\mathrm{d}x=|\det \tA |\cdot |\Om_1|> C_{\eta}.
  \ee
  This and the fact $\tA\Om_1 \subset \widetilde\Om_1=A(y)\Om\cap B_1$ imply that
  \be
  |A(y)\Om\cap B_1|=|\widetilde\Om_1|\geq |\tA\Om_1|> C_{\eta}.
  \ee
  By using the spherical coordinates, we write
  \be
  \begin{split}
  |\widetilde\Om_1|=
  \int_0^1\sigma(\widetilde\Om_1\cap \bbS^{n-1})r^{n-1}dr=\frac{\sigma(\widetilde\Om_1\cap \bbS^{n-1})}{n}.
  \end{split}
  \ee
  This and the fact $A(y)\Om\cap \bbS^{n-1}=\widetilde\Om_1\cap \bbS^{n-1}$ imply that
  \ben\label{for proof, 15}
  \sigma(A(y)\Om\cap \bbS^{n-1})=\sigma(\widetilde\Om_1\cap \bbS^{n-1})> C_{\eta}.
  \een
  Take $\{\varphi_{\ell}\}_{\ell=1}^{\infty}$ to be sequence of nonnegative $C_c^{\infty}$ functions supported in $\Omega$,
  satisfying $0\leq \varphi_i \leq\varphi_\ell\leq 1$ if $i<\ell$ for all $i,\ell\in \mathbb{Z}^+$,
  and $\lim\limits_{{\ell}\rightarrow \infty}\varphi_{\ell}=\chi_{\Om}$.

  It follows by (\ref{derivative of Hausdorff operator for alpha=1, 2, matrix}), (\ref{for proof, 14}), (\ref{for proof, 15})
  and the Lebesgue's monotone convergence theorem
  that for $1\leq j\leq n$,
    \begin{equation}\label{for proof, 2, matrix}
    \begin{split}
    &\lim_{l\rightarrow \infty}\int_{\mathbb{R}^n}\frac{\partial H_{\Phi, A}G_1}{\partial x_j}(x)\varphi_l(x) \mathrm{d}x
    \\
      &=
      \lim_{l\rightarrow \infty}
      \int_{\mathbb{R}^n}\Phi(y)
      \int_{\mathbb{R}^n} \varphi_l(x)\cdot \sum_{i=1}^n a_{i j}(y) \frac{\partial G_1}{\partial x_{i}}\left( A(y)x \right)\mathrm{d}x \mathrm{d}y
      \\
      &=
      \int_{\mathbb{R}^n} \Phi(y) \int_{\Om} \sum_{i=1}^n a_{i j}(y) \frac{\partial G_1}{\partial x_{i}}\left( A(y)x \right)\mathrm{d}x \mathrm{d}y
      \\
      &=
      \int_{\mathbb{R}^n} \Phi(y)\left|\det A(y)\right|^{-1}
      \int_{A(y)\Om} \sum_{i=1}^n a_{i j}(y) \frac{\partial G_1}{\partial x_{i}}( x )\mathrm{d}x\mathrm{d}y
      \\
      &\geq
      \int_{\mathbb{R}^n} \Phi(y)\left|\det A(y)\right|^{-1}
      \int_{A(y)\Om} \sum_{i=1}^n a_{i j}(y) 2\mathrm{e}^{-2|x|^2-\frac{n}{2}}\mathrm{d}x\mathrm{d}y
      \\
      &\sim
      \int_{\mathbb{R}^n} \Phi(y)\left|\det A(y)\right|^{-1}\sum_{i=1}^n a_{i j}(y)\int_0^{\infty}e^{-2r^2}r^{n-1}\sigma(A(y)\Om\cap \bbS^{n-1})\mathrm{d}r\mathrm{d}y
      \\
      &\gtrsim
      \int_{\mathbb{R}^n} \Phi(y)\left|\det A(y)\right|^{-1}\sum_{i=1}^n a_{i j}(y)\mathrm{d}y.
    \end{split}
  \end{equation}
  By the fact $H_{\Phi,A}G_1\in W^{k,1}$, we further obtain that for $1\leq j\leq n$,
  \be
  \begin{split}
  \int_{\mathbb{R}^n} \Phi(y)\left|\det A(y)\right|^{-1}\sum_{i=1}^n a_{i j}(y)\mathrm{d}y
  \lesssim &
  \lim_{l\rightarrow \infty}\int_{\mathbb{R}^n}\frac{\partial H_{\Phi, A}G_1}{\partial x_j}(x)\varphi_l(x) \mathrm{d}x
  \\
  \lesssim &
  \lim_{l\rightarrow \infty}\left\|\frac{\partial H_{\Phi, A}G_1}{\partial x_j}\right\|_{L^1}\cdot \|\varphi_l\|_{L^{\infty}}
  \\
  \lesssim &
  \lim_{l\rightarrow \infty}\left\|H_{\Phi, A}G_1\right\|_{W^{k,1}}\cdot \|\varphi_l\|_{L^{\infty}}.
  \end{split}
  \ee
This implies that
  \begin{equation}\label{for proof, 3, one order, all index, matrix}
  \begin{split}
    \int_{\mathbb{R}^n}\Phi(y)\left|\det A(y)\right|^{-1} \left(1+\|A(y)\|\right) \mathrm{d}y
    \lesssim
    \sum_{j=1}^n\int_{\mathbb{R}^n} \Phi(y)\left|\det A(y)\right|^{-1}\sum_{i=1}^n a_{i j}(y)\mathrm{d}y
    \lesssim 1.
  \end{split}
  \end{equation}
Moreover, by Theorem \ref{theorem, boundedness, matrix}, we have the following expression of weak derivative:
  \begin{equation*}
    \frac{\partial H_{\Phi,A}f}{\partial x_j}(x)
    =\int_{\mathbb{R}^n}\Phi(y)\cdot \sum_{i=1}^n a_{i j}(y) \frac{\partial f}{\partial x_{i}}\left( A(y)x \right)\mathrm{d}y,
    \ \ \ (j=1,2,\cdots,n).
  \end{equation*}
  \\
  \textbf{Step 2: high order derivative with $k\geq2$.}\\
  Let $2\leq m\leq k$ be an integer.
  For a multi-index $\al$ with $|\al|=m$, there exist $s\in \{1,2,\cdots,n\}$ and a multi-index $\b$ such that
  \be
  |\b|=m-1,\ \ \ \b_s=\al_s-1,\ \ \ \b_i=\al_i,\ (i\neq s), \text{ i.e. } D^{\al}=D^{\b}\cdot\frac{\partial}{\partial x_s}.
  \ee
  Assume by induction that
  \begin{equation}\label{for proof, 4, matrix}
      \int_{\mathbb{R}^n}\Phi(y)\left|\det A(y)\right|^{-1} \left(1+\|A(y)\|^{m-1}\right) \mathrm{d}y <\infty,
    \end{equation}
    and
    \begin{align}\label{for proof, 5, matrix}
      &D^{\b}(H_{\Phi, A}f)(x)
      =\int_{\bbR^n}\Phi(y)D^{\b}(f(A(y)x))\mathrm{d}x
      \nonumber
      \\
      &=\int_{\mathbb{R}^n}\Phi(y)\cdot
      \left[\prod_{j=1}^n\left(\sum_{i=1}^na_{ij}(y)\frac{\partial}{\partial x_i}\right)^{\b_j}\right](f)(A(y)x)
        \mathrm{d}y.
    \end{align}
   By the definition of weak derivative and (\ref{for proof, 5, matrix}), for $\varphi \in \mathscr{D}$ and $f\in W_{k,1}$, we have
   \begin{equation}\label{for proof, 6, matrix}
     \begin{split}
       &\int_{\mathbb{R}^n}D^{\alpha}(H_{\Phi, A}f)(x)\varphi(x)\mathrm{d}x
       =
       (-1)^m\int_{\mathbb{R}^n}H_{\Phi, A}f \cdot D^{\b}\frac{\partial\varphi}{\partial x_s}(x)\mathrm{d}x
       \\
       = &
       -\int_{\mathbb{R}^n}D^{\b}(H_{\Phi, A}f)(x)\frac{\partial\varphi}{\partial x_s}(x)\mathrm{d}x
       \\
       = &
       -\int_{\mathbb{R}^n}
       \int_{\mathbb{R}^n}\Phi(y)
          \left[\prod_{j=1}^n\left(\sum_{i=1}^na_{ij}(y)\frac{\partial}{\partial x_i}\right)^{\b_j}\right](f)(A(y)x)\mathrm{d}y
       \frac{\partial\varphi}{\partial x_s}(x)\mathrm{d}x.
     \end{split}
   \end{equation}
   It follows by the induction that
   \ben\label{for proof, 13}
   \begin{split}
     &\left\|\Phi(y)
          \left[\prod_{j=1}^n\left(\sum_{i=1}^na_{ij}(y)\frac{\partial}{\partial x_i}\right)^{\b_j}\right](f)(A(y)x)
          \frac{\partial\varphi}{\partial x_s}(x)
       \right\|_{L^1(\bbR^n\times \bbR^n)}
       \\
       \lesssim &
       \left\|\frac{\partial\varphi}{\partial x_s}\right\|_{L^{\infty}}
       \int_{\bbR^n} \Phi(y)
       \left[\prod_{j=1}^n\left(\sum_{i=1}^na_{ij}(y)\right)^{\b_j}\right]
       \sum_{|\g|=m-1}\|D^{\g}(f)(A(y)x)\|_{L^1}\mathrm{d}y
       \\
       \lesssim &
       \int_{\bbR^n} \Phi(y)
       \|A(y)\|^{m-1}\sum_{|\g|=m-1}\|D^{\g}(f)(A(y)x)\|_{L^1}\mathrm{d}y
       \\
       \lesssim &
       \int_{\bbR^n}\Phi(y)|\det A(y)|^{-1}\|A(y)\|^{m-1}\mathrm{d}y<\infty.
   \end{split}
   \een
   Applying Fubini theorem to (\ref{for proof, 6, matrix}), we obtain that
   \begin{equation}\label{for proof, 7, matrix}
     \begin{split}
     &\int_{\mathbb{R}^n}D^{\alpha}(H_{\Phi, A}f)(x)\varphi(x)\mathrm{d}x
       \\
       = &
       -\int_{\mathbb{R}^n}\Phi(y)
       \int_{\mathbb{R}^n}
          \left[\prod_{j=1}^n\left(\sum_{i=1}^na_{i,j}(y)\frac{\partial}{\partial x_i}\right)^{\b_j}\right](f)(A(y)x)
       \frac{\partial\varphi}{\partial x_s}(x)\mathrm{d}x \mathrm{d}y
       \\
       = &
       \int_{\mathbb{R}^n}\Phi(y)
       \int_{\mathbb{R}^n}
          \left\{
          \left(\sum_{\ell=1}^n a_{\ell s}\frac{\partial}{\partial x_{\ell}}\right)\cdot
          \left[\prod_{j=1}^n\left(\sum_{i=1}^na_{ij}(y)\frac{\partial}{\partial x_i}\right)^{\b_j}\right]
          \right\}
          (f)(A(y)x)
       \varphi(x)\mathrm{d}x \mathrm{d}y,
       \\
       = &
       \int_{\mathbb{R}^n}\Phi(y)
       \int_{\mathbb{R}^n}
          \left[\prod_{j=1}^n\left(\sum_{i=1}^na_{ij}(y)\frac{\partial}{\partial x_i}\right)^{\al_j}\right]
          (f)(A(y)x)
       \varphi(x)\mathrm{d}x \mathrm{d}y,
     \end{split}
   \end{equation}
   where in the last equality we use the definition of weak derivative and the fact that $f\in W^{k,1}$.
  Set \[G_m(x)=(-1)^m\prod\limits_{l=1}^{n} g(x_l+a_m).\]
  Using Lemma \ref{lemma, Derivatives of Gaussian function},  for Gaussian function $g$ we have
  $g^{(l)}(t)=P_l(t)g(t)$, where $P_l(t)$ is a $l$-order polynomial of $t$.
  Denote $a_m$ the biggest one of all the positive roots of $(-1)^iP_i-1(i=1,2,\cdots,m)$ if exists, else $a_m=0$.
  Then Lemma \ref{lemma, Derivatives of Gaussian function} yields that for multi-index $\g$ with $|\g|=m$,
  \begin{equation*}
    D^{\g} G_m(x)
    =(-1)^m \prod_{l=1}^{n}P_{|\g_l|}(x_l+a_m)g(x_l+a_m)
    = \prod_{l=1}^{n}(-1)^{|\g_l|} P_{|\g_l|}(x_l+a_m)g(x_l+a_m),
  \end{equation*}
  where the sign of the leading term of $(-1)^{|\g_l|}P_{|\g_l|}$ is positive.
  By the choice of $a_m$, for every multi-index $\g$ with $|\g|=m$,
  we have
  \be
  (-1)^{|\g_l|} P_{|\g_l|}(x_l+a_m)\geq 1,\ \ \  \text{for}\ x\in \Omega:=(0,\infty)^n,
  \ee
  and
  $$D^{\g} G_m(x)\geq \prod_{l=1}^{n}g(x_l+a_m),\ \ \  \text{for}\ x\in \Omega:=(0,\infty)^n.$$
This implies that for $x\in\Om$,
\ben\label{for proof, 12}
\left(\prod_{j=1}^n\left(\sum_{i=1}^na_{ij}(y)\frac{\partial}{\partial x_i}\right)^{\al_j}\right)
          (G_m)(x)\geq
          \left(\prod_{j=1}^n\left(\sum_{i=1}^na_{ij}(y)\right)^{\al_j}\right)\prod_{l=1}^{n}g(x_l+a_m)>0.
\een
From this, (\ref{for proof, 7, matrix}) and the Lebesgue monotone convergence theorem, we have
\begin{equation*}
    \begin{split}
      &
      \lim_{l\rightarrow \infty}\int_{\mathbb{R}^n}D^{\alpha}(H_{\Phi, A}G_m)(x)\varphi_l(x)\mathrm{d}x
      \\
      =&
      \lim_{l\rightarrow \infty}
      \int_{\mathbb{R}^n}\Phi(y)
       \int_{\mathbb{R}^n}
          \left[\prod_{j=1}^n\left(\sum_{i=1}^na_{ij}(y)\frac{\partial}{\partial x_i}\right)^{\al_j}\right]
          (G_m)(A(y)x)
       \varphi_l(x)\mathrm{d}x \mathrm{d}y
       \\
       =&
       \int_{\mathbb{R}^n}\Phi(y)
       \int_{\Om}
          \left[\prod_{j=1}^n\left(\sum_{i=1}^na_{ij}(y)\frac{\partial}{\partial x_i}\right)^{\al_j}\right]
          (G_m)(A(y)x)\mathrm{d}x \mathrm{d}y
       \\
       =&
       \int_{\mathbb{R}^n}\Phi(y) \cdot |\det A(y)|^{-1}
       \int_{A(y)\Om}
          \left[\prod_{j=1}^n\left(\sum_{i=1}^na_{ij}(y)\frac{\partial}{\partial x_i}\right)^{\al_j}\right]
          (G_m)(x)\mathrm{d}x \mathrm{d}y,
       \end{split}
      \end{equation*}
where $\{\varphi_l\}_{l=1}^{\infty}$ is a sequence of functions as in the Step 1.
Next, by (\ref{for proof, 12}),
we further obtain that
\begin{align*}
       &
      \lim_{l\rightarrow \infty}\int_{\mathbb{R}^n}D^{\alpha}(H_{\Phi, A}G_m)(x)\varphi_l(x)\mathrm{d}x
              \\
       &\geq
       \int_{\mathbb{R}^n}\Phi(y) \cdot |\det A(y)|^{-1}
       \int_{A(y)\Om}
          \left(\prod_{j=1}^n\left(\sum_{i=1}^na_{ij}(y)\right)^{\al_j}\right)\prod_{l=1}^{n}g(x_l+a_m)\mathrm{d}x \mathrm{d}y
       \\
       &\geq\int_{\mathbb{R}^n}\Phi(y) \cdot |\det A(y)|^{-1}\left(\prod_{j=1}^n
         \left(\sum_{i=1}^na_{ij}(y)\right)^{\al_j}\right)\mathrm{d}y
       \inf_{y\in \bbR^n}\int_{A(y)\Om}\prod_{l=1}^{n}g(x_l+a_m)\mathrm{d}x.
  \end{align*}
Observe that
\be
g(x_l+a_m)=e^{-|x_l+a_m|^2}\geq e^{-2(|x_l|^2+|a_m|^2)}=e^{-2|a_m|^2}e^{-2|x_l|^2}.
\ee
Recall the following fact proved in Step 1
 \be
  \mathrm{\sigma}\left(\left(A(y)\Om\right)\cap \bbS^{n-1}\right) > C_{\eta} ,\ \ \ y\in \bbR^n.
 \ee
We obtain that for every $y\in \bbR^n$,
\be
\begin{split}
\int_{A(y)\Om}\prod_{l=1}^{n}g(x_l+a_m)\mathrm{d}x
\geq &
\int_{A(y)\Om}\prod_{l=1}^{n}e^{-2|a_m|^2}e^{-2|x_l|^2}\mathrm{d}x
\\
= &
e^{-2n|a_m|^2}\int_{A(y)\Om}\prod_{l=1}^{n}e^{-2|x_l|^2}\mathrm{d}x
\\
= &
e^{-2n|a_m|^2}\int_{A(y)\Om}e^{-2|x|^2}\mathrm{d}x
\\
= &
e^{-2n|a_m|^2}\int_0^{\infty}e^{-2r^2}r^{n-1}\mathrm{\sigma}\left(\left(A(y)\Om\right)\cap \bbS^{n-1}\right)dr
\gtrsim 1.
\end{split}
\ee
Combining with the above estimates, we obtain
\be
\begin{split}
\lim_{l\rightarrow \infty}\int_{\mathbb{R}^n}D^{\alpha}(H_{\Phi, A}G_m)(x)\varphi_l(x)\mathrm{d}x
       \gtrsim
       \int_{\mathbb{R}^n}\Phi(y) \cdot |\det A(y)|^{-1}
       \left[\prod_{j=1}^n
         \left(\sum_{i=1}^na_{ij}(y)\right)^{\al_j}\right]\mathrm{d}y.
\end{split}
\ee
Hence, combing with the fact $H_{\Phi, A}G_m\in W^{k,1}$, we have
\be
\begin{split}
  &\int_{\mathbb{R}^n}\Phi(y) \cdot |\det A(y)|^{-1}
        \left[\prod_{j=1}^n
         \left(\sum_{i=1}^na_{ij}(y)\right)^{\al_j}\right]\mathrm{d}y
         \\
         \lesssim &
         \lim_{l\rightarrow \infty}\int_{\mathbb{R}^n}|D^{\alpha}(H_{\Phi, A}G_m)(x)|\cdot |\varphi_l(x)|\mathrm{d}x
         \\
         \lesssim &
         \lim_{l\rightarrow \infty}\|D^{\alpha}(H_{\Phi, A}G_m)(x)\|_{L^1}\cdot \|\varphi_l(x)\|_{L^{\infty}}
         \lesssim
         \|H_{\Phi,A}G_m\|_{W^{k,1}}.
\end{split}
\ee
Note that the above inequality is valid for all $\al$ with $|\al|=m$. We obtain that
 \be
 \begin{split}
 &\int_{\mathbb{R}^n}\Phi(y) \cdot |\det A(y)|^{-1}
         (1+\|A\|^m))\mathrm{d}y
         \\
 \lesssim &
 \sum_{|\al|=m}
  \int_{\mathbb{R}^n}\Phi(y) \cdot |\det A(y)|^{-1}
       \left[\prod_{j=1}^n
         \left(\sum_{i=1}^na_{ij}(y)\right)^{\al_j}\right]\mathrm{d}y\lesssim 1.
 \end{split}
 \ee
   By induction, the conclusion is valid for all $m$ with $2\leq m\leq k$.
  This completes our proof.
\end{proof}

\section{Applications}
As we mentioned before, Hausdorff operators can be regarded as the generalization of some classical operators, such as Hardy operator $H$ and its adjoint operator $H^*$.
Thus, by choosing special $\Phi$, we can obtain the bounded and unbounded properties for special operator.
In one dimension, we take $\Psi(t)=\frac{\chi_{(1,\infty)}(t)}{t^2}$ and $\Psi^*(t)=\frac{\chi_{(0,1)}(t)}{t}$, we have
\begin{equation*}
  H_{\Psi}f(x)=Hf(x)=\frac{1}{x}\int_0^x f(t)dt
\end{equation*}
and
\begin{equation*}
  H_{\Psi^*}f(x)=H^*f(x)=\int_x^{\infty} \frac{f(t)}{t}dt
\end{equation*}
respectively.
\begin{proposition}
  Hardy operator $H$ is not bounded on $W^{k,1}(\mathbb{R})$ for all $k\geq0$.
  The adjoint operator of Hardy operator $H^*$ is bounded on on $L^1(\mathbb{R})$, but not bounded on $W^{k,1}(\mathbb{R})$ with $k\in\mathbb{Z}^+$.
\end{proposition}
\begin{proof}
  For the boundedness of Hardy operator on Sobolev spaces, by Corollary \ref{corollary, boundedness of Hausdorff operator on Sobolev space L_k^1}, we only need to check whether (\ref{condition of main theorem}) holds for $\Psi$.
  A direct computation yields that
  \begin{equation}
  \int_{\mathbb{R}} \Psi(t)\cdot |t|(1+|t|^{-k})dt
  =
    \int_{\mathbb{R}} \frac{\chi_{(1,\infty)}(t)}{t^2}\cdot |t|(1+|t|^{-k})dt
    \geq\int_1^{+\infty}\frac{1}{t}dt=+\infty.
  \end{equation}
  Thus, Hardy operator is not bounded on $W^{k,1}(\mathbb{R})$.

  To verify the boundedness of the adjoint operator of Hardy operator, we need to check (\ref{condition of main theorem}) for $\Psi^*$.
  More precisely, we have
  \begin{equation}
    \int_{\mathbb{R}} \Psi^*(t)\cdot |t|dt
  =\int_{\mathbb{R}} \frac{\chi_{(0,1)}(t)}{t}\cdot |t|dt=1
  \end{equation}
  and when $k$ is a positive integer
  \begin{equation}
  \int_{\mathbb{R}} \Psi^*(t)\cdot |t|(1+|t|^{-k})dt
  =
    \int_{\mathbb{R}} \frac{\chi_{(0,1)}(t)}{t}\cdot |t|(1+|t|^{-k})dt
    \geq\int_0^1 t^{-k}dt=+\infty.
  \end{equation}
  The desired conclusion follows by using Corollary \ref{corollary, boundedness of Hausdorff operator on Sobolev space L_k^1}.
\end{proof}

\begin{remark}
A classical result shows that Hardy-Littlewood maximal operator is not bounded on $L^1$.
However, the boundedness of first derivative of Hardy-Littlewood maximal operator is proved to be true at
certain endpoint spaces (see \cite{CMP_JFA_2017}).
Like the maximal operator, Hardy operator is not bounded on $L^1$,
however, one can verify its boundedness on $BV(\bbR)$.
\end{remark}

\subsection*{Acknowledgements} Zhao's work was partially supported by the National Natural Foundation of China (Nos. 11601456, 11771388)
and Guo's work was partially supported by the National Natural Foundation of China (Nos. 11701112, 11671414) and the China postdoctoral Science Foundation (No. 2017M612628).

\end{document}